\documentclass[12pt]{article}

\usepackage[english]{babel}
\usepackage[cp1251]{inputenc}
\usepackage[T2A]{fontenc}

\usepackage{amssymb}
\usepackage{amsmath}
\usepackage{amsthm} 
\usepackage{mathtext}
\usepackage{mathtools}
\usepackage{amsfonts}
\usepackage{xcolor}

\renewcommand{\leq}{\leqslant}
\renewcommand{\geq}{\geqslant}

\newcommand{\refpar}{Sect.~}

\newcommand{\lda}{\lambda}
\newcommand{\Wo}{{\raisebox{0.2ex}{\(\stackrel{\circ}{W}\)}}{}}

\DeclareSymbolFont{myletters}{OML}{ztmcm}{m}{it}
\DeclareMathSymbol{\uplambda}{\mathord}{myletters}{"15}

\newtheorem{theorem}{Theorem}
\newtheorem{lemma}{Lemma}
\newtheorem{proposition}{Proposition}

\newtheorem{conjecture}{Conjecture}

\theoremstyle{definition}
\newtheorem{definition}{Definition}
\newtheorem{remark}{Remark}

\newenvironment{enbibliography}{\vspace{-0.5cm}\begin{thebibliography}}{\end{thebibliography}}

\begin{document} 
\title{On the spectrum of the Sturm--Liouville 
problem with arithmetically self-similar weight}
\author{N.~V.~Rastegaev \\
\small{St. Petersburg State University} \\ 
\small{7/9 Universitetskaya nab., St. Petersburg, 199034 Russia} 
\\ \small{rastmusician@gmail.com}}
\renewcommand{\today}{}
\maketitle
\abstract{
Spectral asymptotics of the Sturm--Liouville problem with an arithmetically 
self-similar singular weight is considered. In previous papers 
A. A. Vladimirov and I. A. Sheipak, as well as the author, rely on 
the spectral periodicity property, which places major constraints on the self-similarity parameters of the weight. In this study, a different approach to eigenvalue counting function estimation is presented. As a result, a significantly wider class of self-similar measures can be considered. The obtained asymptotics is applied to the small ball deviations problem for the Green Gaussian processes.
}

\section{Introduction}
We generalize the results of 
\cite{VSh3}, \cite{Rast1} on the spectral asymptotics of the problem
\begin{gather}
\left\{
\begin{split}\label{eq:1.1}
    &-y'' = \lambda\mu y,\\
    &y'(0) = y'(1) = 0,
\end{split}
\right.
\end{gather}
where the weight measure $\mu$ is a distributional derivative of a self-similar generalized Cantor type function
(in particular, $\mu$ is singular with respect to the Lebesgue measure). 
\begin{remark}\label{remark1}
It is well known that the change of the boundary conditions causes a rank two perturbation of the quadratic form corresponding to the problem \eqref{eq:1.1}. It follows from the general variational theory 
(see \cite[\S10.3]{BS2}) that counting functions of the eigenvalues of boundary-value problems related to the same equation but different boundary conditions cannot differ by more than 2. Thus, the main term of spectral asymptotics does not depend on the boundary conditions.

Also, it follows from \cite[Theorem 3.2]{Mats} (see also \cite[Lemma 5.1]{NazSheip} for a simple variational proof), that relatively compact perturbations of the operator (e.g. lower order terms) do not affect the main term of the asymptotics given by \eqref{eq:mes_asymp} below.
\end{remark}
\begin{remark}
Spectral asymptotics of the problem \eqref{eq:1.1}, aside from being interesting in itself, arises in the problem of small ball deviations of Green Gaussian processes in $L_2(\mu)$ (see \cite{Naz}).
\end{remark}

The problem of the eigenvalues asymptotic behavior for the problem \eqref{eq:1.1} goes back to the works of
M.~G.~Krein (see, e.g., \cite{K}).

From \cite{BS1} it follows that if the measure $\mu$ contains absolutely
continuous component, its singular component does not influence
the main term of the spectral asymptotic.

In the case of singular measure $\mu$ it follows from early works by M.~G.~Krein, that 
the counting function $N:(0,+\infty)\to\mathbb N$ of eigenvalues of the problem \eqref{eq:1.1} 
admits the estimate $o(\lda^{\frac{1}{2}})$ instead of the usual asymptotics $N(\lda)\sim C\lda^{\frac{1}{2}}$ in the case of measure containing a regular component. (see, e.g., \cite{KrKac} or \cite{McKeanRay}, and also \cite{B} for similar results for higher even order operators and better lower bounds for eigenvalues for some special classes of measures).

Exact power exponent $D$ of the counting function $N(\lda)$ in
the case of self-similar measure $\mu$ was established in \cite{F} (see also earlier works \cite{HU} and \cite{McKeanRay} for particular results,
concerning the classical Cantor ladder, and \cite{RFr} for a generalization to the case of self-conformal, i.e. self-similar via non-affine contractions, measures).

It is shown in \cite{SV} and \cite{KL} that the eigenvalues counting function of problem \eqref{eq:1.1} has the asymptotics
\begin{equation}\label{eq:mes_asymp}
	N(\lambda)=\lambda^D\cdot\bigl(s(\ln\lambda)+o(1)\bigr),
	\qquad \lambda\to+\infty,
\end{equation}
where $D\in(0,\frac{1}{2})$ and $s$ is a continuous function.
In the case of non-arithmetic type of self-similarity (see Definition \ref{arithm} below) of the 
primitive for $\mu$, the function $s$ degenerates into constant. In the case of arithmetic self-similarity, it has a period $T$, which depends on the parameters of the self-similarity (see also \cite{VSh1} for similar results in the case of indefinite weight).

In the paper \cite{Naz} this result is generalized to the case of an arbitrary even order differential operator. Also the following conjecture is introduced.
\begin{conjecture}\label{Conj1} 
Function $s$ is not constant for arbitrary non-constant weight with arithmetically self-similar primitive.
\end{conjecture}

The paper \cite{VSh1} gives computer-assisted proof of this conjecture in the simplest case, when the generalized primitive of weight $\mu$ is a classical Cantor ladder.

In \cite{VSh3} Conjecture \ref{Conj1} was confirmed for ``even'' ladders (see Definition \ref{even} below). For such ladders the following theorem was proved.
\begin{theorem}\label{Th1}
The coefficient $s$ from the asymptotic \eqref{eq:mes_asymp} satisfies the relation
\begin{equation}\label{Th1eqn}
	\forall t\in [0,T]\quad s(t)=e^{-Dt}\,\sigma(t),
\end{equation}
where $\sigma$ is some purely singular non-decreasing function, that is, the primitive of a measure singular with respect to the Lebesgue measure.
\end{theorem}
Hence the relation $s(t)\neq const$ follows immediately.
This result is generalized in \cite{V3} to the case of 
the fourth order operators.

In paper \cite{Rast1} the result of \cite{VSh3} is generalized to the wider class of ladders satisfying conditions \eqref{wider_case}.

The aim of this paper is to generalize Theorem \ref{Th1} to the case of an arbitrary arithmetically self-similar ladder with non-empty intermediate intervals.

This paper has the following structure. \refpar{2} provides the necessary definitions of self-similar functions of generalized Cantor type, derives their properties, defines the classes of functions under consideration and states the main result. \refpar{3}
establishes the auxiliary facts, concerning the spectral properties. In \refpar{4} Theorem \ref{Th1} is proved for the suggested class of ladders. Finally, in \refpar{5} we provide the link between the results of this paper and the problem of small ball deviations of Gaussian processes.

\section{Self-similar functions. Main result statement}\label{par:2}
\noindent
Let $m\geq 2$, and let
$\{I_k = [a_k,b_k]\}_{k=1}^m$ be sub-segments of $[0,1]$, without interior intersection, i.e. $b_j \leqslant a_{j+1}$ for all $j=1\ldots m-1$. Next, let positive values $\{\rho_i\}_{i=1}^m$ satisfy the relation $\sum\limits_{k=1}^m\rho_k = 1$ and let $\{e_i\}_{i=1}^m$ be boolean values. 

We define a family of affine mappings
\begin{gather*}
S_i(t) = 
\left\{
\begin{split}
&a_i + (b_i-a_i)\,t, \quad e_i = 0, \\
&b_i - (b_i-a_i)\,t, \quad e_i = 1,
\end{split}
\right.
\end{gather*}
contracting $[0,1]$ onto
$I_i$ and changing the orientation when $e_i = 1$.

We define the operator $\mathcal{S}$ on the space $L_{\infty}(0,1)$ as follows:
\begin{equation*}
\mathcal{S}(f) = \sum\limits_{i=1}^m
\left((e_i + (-1)^{e_i}f\circ S_i^{-1})\cdot \chi_{I_i}+\chi_{\{x>b_i\}} \right)\rho_i,
\end{equation*}
where $\chi$ stands for the indicator function of a set.
Thus, the graph of the function $\mathcal{S}(f)$ on each segment $I_i$ is an appropriately shrunk graph of the function $f$. On any intermediate interval $\mathcal{S}(f)$ is constant.

\begin{proposition}
\textbf{\textsc{(see, e.g. \cite[Lemma 2.1]{Sh})}}
$\mathcal{S}$ is a contraction mapping on $L_{\infty}(0,1)$.
\end{proposition}
Hence, by the Banach fixed-point theorem there exists a (unique) function 
$\mathcal{C}\in L_{\infty}(0,1)$ such that $\mathcal{S}(\mathcal{C})=\mathcal{C}$.
Such a function $\mathcal{C}(t)$ will be called the \textit{generalized Cantor ladder} with $m$ steps.

The function $\mathcal{C}(t)$ can be found as the uniform limit of the sequence $\mathcal{S}^k(f)$ for \mbox{$f(t)\equiv t$}, which allows us to assume that it is continuous. It is also easy to demonstrate that it is monotone, $\mathcal{C}(0)=0$, $\mathcal{C}(1)=1$.
The derivative of  $\mathcal{C}(t)$ in the sense of distributions is a singular measure $\mu$ without atoms, self-similar in the sense of Hutchinson (see \cite{H}), i.e. it satisfies the relation
\begin{equation*}
\mu(E) = \sum\limits_{k=1}^m \rho_k \cdot \mu(S_k^{-1}(E\cap I_k))
\end{equation*}
for arbitrary measurable set $E$.
More general constructions of self-similar functions are described in
\cite{Sh}.

\begin{remark}
Without loss of generality we could assume that $a_1 = 0$, $b_m = 1$, otherwise the measure could be stretched, which leads to the spectrum being multiplied by a constant.
\end{remark}

\begin{definition}\label{arithm}
The self-similarity is called \textit{arithmetic} if
the logarithms of the values $\rho_k(b_k-a_k)$ are commensurable.
In other words,
\begin{equation*}
\rho_i(b_i-a_i) = \tau^{k_i}, \quad i=1,\ldots, m,
\end{equation*}
for a certain constant $\tau$ and $k_i \in \mathbb{N}$, such that $\text{GCD}(k_i, i=1,\ldots, m) = 1$.
\end{definition}

\noindent We call the generalized Cantor ladder \textit{even}, if
\begin{equation}\label{even}
 \forall i=2,\dots, m\quad \rho_i=\rho_1=\frac{1}{m},
\quad b_i - a_i=b_1 - a_1,\quad a_i - b_{i-1}=a_2 - b_1 > 0.
\end{equation}
That is the class of ladders considered in \cite{VSh3}.

In \cite{Rast1} the relation \eqref{Th1eqn} is proved for arithmetically self-similar ladders with the following conditions:
\begin{equation}\label{wider_case}
\forall i = 2,\dots, m \quad k_i = k_1 = 1, \;\;
a_{i}-b_{i-1} > 0.
\end{equation}

Let's state the main result of this paper.
\begin{theorem}\label{main_theo}
Let the ladder be arithmetically self-similar, and let $a_{i}-b_{i-1} > 0$ for all $i=2,\ldots, m$.
Then the formula \eqref{Th1eqn} holds.
\end{theorem}

\begin{remark}
For the described class of ladders the power exponent $D$ and the period of the function $s(t)$ are defined by the following relations, obtained in \cite{SV}:
\begin{equation}\label{period}
\sum\limits_{i=1}^m \tau^{k_iD} = 1, \quad T = -\ln\tau.
\end{equation}
\end{remark}

\section{Auxiliary facts about the spectrum}\label{par:3}
We consider the formal boundary value problem on a segment $[a,b] \subset [0,1]$:
\begin{gather}
\left\{
\begin{split}\label{eq:3.1}
    &-y''=\lambda\mu y,\\ 
    &y'(a)-\gamma_0 y(a)=y'(b)+\gamma_1 y(b)=0.
\end{split}
\right.
\end{gather}
We call the function $y \in W_2^1[a,b]$ its generalized solution if it satisfies the integral identity
$$ \int\limits_a^b y'\eta'\, dx +
\gamma_0y(a)\eta(a)+\gamma_1y(b)\eta(b) = 
\lda\int_a^b y\eta \;\mu(dx)  $$
for any $\eta \in W_2^1[a,b]$. Substituting functions \mbox{$\eta \in \Wo_2^1[a,b]$} into the integral identity, we establish that the derivative $y'$ is a primitive of a singular signed measure without atoms $-\lda\mu y$, whence $y$ is continuously differentiable. 

Hereinafter, a particular case of \cite[Proposition~11]{V2} is required:

\begin{proposition}\label{oscil} 
Let $\{\lambda_n\}_{n=0}^{\infty}$ be a sequence
of eigenvalues of the boundary value
problem \eqref{eq:3.1} numbered in ascending order. Then, regardless of the choice of index $n\in\mathbb N$, eigenvalue $\lambda_n$ is
simple, and the corresponding eigenfunction does not vanish on the boundary of the segment \([a,b]\) and has exactly $n$ different zeros within this segment.
\end{proposition}

Let us denote by $\lda_n([a,b])$, $n\geqslant 0$, the eigenvalues of the problem
\begin{gather*}
\left\{
\begin{split}
&-y'' = \lda\mu y, \\
&y'(a) = y'(b) = 0,
\end{split}
\right.
\end{gather*}
and by 
\begin{equation*}
N(\lda, [a,b]) = \#\{ n: \lda_n([a,b])<\lda \} 
\end{equation*}
their counting function. Note, that $\lda_0([a,b]) = 0$.

The following relations follow from the self-similarity of the measure $\mu$.
\begin{lemma}\label{lemma1}
$$\lda_n(I_i) = \tau^{-k_i}\lda_n([0,1]),$$
$$N(\lda, I_i) = N(\tau^{k_i}\lda, [0,1]).$$
\end{lemma}
\begin{proof}
These two relations are equivalent. To prove the first one let's consider the eigenfunction $y_n$ corresponding to the eigenvalue $\lda_n(I_i)$ and let's define function $z$ on $[0,1]$ as
\[
z = y_n \circ S_i,
\]
where $S_i$ is an affine contraction defined in \S 2. It's clear, that the function $z$ satisfies the Neumann boundary conditions on $[0,1]$, and the following relation holds
\[
z'' = (y_n''\circ S_i) \cdot (b_i-a_i)^2 = \lda_n(I_i)(b_i-a_i)^2 \cdot (\mu\circ S_i)\cdot (y_n\circ S_i) .
\]
Note also, that
\[
\mathcal{C}\circ S_i = \mathcal{S}(\mathcal{C}) \circ S_i = \rho_i \cdot (e_i + (-1)^{e_i}\mathcal{C}) + \sum\limits_{j=1}^{i-1} \rho_j,
\]
whence, by differentiating we obtain
\[
\mu\circ S_i = \rho_i(b_i-a_i)^{-1}\mu,
\]
Thus,
\[
z'' = \lda_n(I_i)\rho_i(b_i-a_i) \mu z = \lda_n(I_i)\tau^{k_i} \mu z.
\] 
Thereby, function $z$ corresponds to the eigenvalue $\lda_n(I_i)\tau^{k_i}$ of the Neumann problem on $[0,1]$ and has exactly $n$ zeroes on it, thus, the proof is complete.
\end{proof}

We now prove the main statement of this section.

\begin{theorem}\label{4.1}
Let $J_1 = [c_1, d_1]$, $J_2 = [c_2, d_2]$ 
be subsegments of $[0,1]$, such that $c_2-d_1 \geqslant 0$, and $\mu|_{[d_1, c_2]} \equiv 0$. Denote $J:=[c_1, d_2]$.
Then the function
\begin{equation}\label{Fdef}
	F(\lda) := N(\lda, J) - N(\lda, J_1)-N(\lda, J_2)
\end{equation}
has discontinuities $\lda_n(J)$, $\lda_n(J_1)$, $\lda_n(J_2)$. Further, the elements of  $\{ \lda_n(J) \}_{n=0}^\infty$ and $\{\lda_n(J_1) \}_{n=0}^\infty \cup \{\lda_n(J_2) \}_{n=0}^\infty$ are non-strictly interlacing beginning with the element of the latter. Moreover, $F$ changes its value from $0$ to $-1$ at the points $\{\lda_n(J_1) \}_{n=0}^\infty \cup \{\lda_n(J_2) \}_{n=0}^\infty$ and from $-1$ to $0$ at the points $\{ \lda_n(J) \}_{n=0}^\infty$ not included in $\{\lda_n(J_1) \}_{n=0}^\infty \cup \{\lda_n(J_2) \}_{n=0}^\infty$.
\end{theorem}
\begin{proof}
Consider the quadratic form
\[
Q_1(y,y) := \int\limits_{J} |y'|^2 dt, \quad \mathcal{D}(Q_1) = \big\{y\in W_2^1(J) : y \text{ is linear on } [d_1, c_2]  \big\}.
\]
We recall (see, e.g. \cite[\S10.2]{BS2}), that the counting function $N(\lambda, J)$ could be expressed in terms of this quadratic form:
\[
N(\lambda, J) = \sup \dim \big\{ \mathcal{H}\subset\mathcal{D}(Q_1) : Q_1(y, y) < \lambda\int\limits_{J} y^2(t)\mu(dt) \text{ on } \mathcal{H} \big\}.
\]
Similarly, if we consider the quadratic form
\[
Q_2(y,y) := \int\limits_{J_1} |y'|^2 dt + \int\limits_{J_2} |y'|^2 dt, \quad \mathcal{D}(Q_2) = \big\{y\in W_2^1(J) : y \text{ is linear on } [d_1, c_2]  \big\},
\]
then
\[
N(\lambda, J_1)+N(\lambda, J_2) = \sup \dim \big\{ \mathcal{H}\subset\mathcal{D}(Q_2) : Q_2(y, y) < \lambda\int\limits_{J} y^2(t)\mu(dt) \text{ on } \mathcal{H} \big\}.
\]
We note, that the quadratic forms differ by a positive term
\[
Q_1(y,y) - Q_2(y,y) = \int\limits_{d_1}^{c_2} |y'| dt,
\]
and coincide on a space of codimension $1$:
\[
Q_1(y,y) = Q_2(y,y) \text{ on } \big\{ y\in W_2^1(J) : y \text{ is constant on } [d_1, c_2] \big\}.
\]
Thus,
\begin{equation}\label{eq:4.1}
-1 \leq F(\lambda) \leq 0.
\end{equation}
Note now, that if some point is a discontinuity of two terms in the right side of \eqref{Fdef}, then it is the discontinuity of the third term as well. For example, let $\lda_n(J)$ be a discontinuity of $N(\lambda, J_1)$. Consider the eigenfunction $y_n$ on $J$ corresponding to $\lambda_n(J)$. Then $y_n|_{J_1}$ is an eigenfunction, corresponding to $\lambda_k(J_1)$ for some $k$, $y'_n(d_1) = 0$, thus $y'_n(c_2)=0$ as well, and $y_n|_{J_2}$ is an eigenfunction of the Neumann problem on $J_2$, which means that $\lambda_n(J)$ is also a discontinuity point of $N(\lambda, J_2)$. Note also, that according to Proposition \ref{oscil} every term changes exactly by $1$ at every discontinuity point.

This implies, that 
$F$ decreases by $1$ at all points of $\{\lda_n(J_1) \}_{n=0}^\infty \cup \{\lda_n(J_2) \}_{n=0}^\infty$ (at each point either only one term changes by $1$, or all three do). By \eqref{eq:4.1} it changes its value from $0$ to $-1$. Obviously, it must change the value from $-1$ to $0$ at all other discontinuities, which are the elements of $\{\lda_n(J) \}_{n=0}^\infty$ not included in $\{\lda_n(J_1) \}_{n=0}^\infty \cup \{\lda_n(J_2) \}_{n=0}^\infty$.
Moreover, no two points from the collection $\{\lda_n(J_1) \}_{n=0}^\infty \cup \{\lda_n(J_2) \}_{n=0}^\infty$ could go in a row without a point from $\{\lda_n(J) \}_{n=0}^\infty$ between them. Similarly, two points form $\{\lda_n(J) \}_{n=0}^\infty$ could not go in succession without a point from $\{\lda_n(J_1) \}_{n=0}^\infty \cup \{\lda_n(J_2) \}_{n=0}^\infty$ between them, which implies, that two collections non-strictly interlace. Since $F(0) = 0$, the interlacing starts with an element of $\{\lda_n(J_1) \}_{n=0}^\infty \cup \{\lda_n(J_2) \}_{n=0}^\infty$.
\end{proof}

\begin{remark}
The proof of Theorem \ref{4.1} does not require $c_2-d_1>0$, so it doesn't use the restriction $a_{i}-b_{i-1}>0$, thus it could be used even when the ladder has empty intermediate intervals. Moreover, the proof uses the eigenfunction oscillation properties (Proposition \ref{oscil}), but does not use the self-similarity of measure $\mu$.

In the case of $m=2$, $k_1=k_2=1$, 
Theorem \ref{4.1} and Lemma \ref{lemma1} imply, that
\[
N(\tau^{-1}\lda_n) = 2N(\lda_n)
\]
and, respectively,
\[
\tau\lda_{2n} = \lda_n.
\]
This relation is called \emph{spectral periodicity} in \cite{VSh3}, \cite{Rast1}, \cite{V3} and \emph{renormalization property} after \cite{KL}.
\end{remark}

\section{Main result proof}\label{par:4}
To prove Theorem \ref{main_theo} we need the following facts:

\begin{proposition}\label{sing}
\textbf{\textsc{(\cite[Proposition~4.1.3]{VSh3})}}
Let  $f\in L_2[0,1]$ be a bounded non-decreasing non-constant function,
let $\{f_n\}_{n=0}^{\infty}$ be a sequence of non-decreasing non-constant step functions and let $\{\mathfrak A_n\}_{n=0}^{\infty}$ be the sequence of  discontinuity point sets of functions $f_n$. Suppose also that the following asymptotic relation holds as  $n\to\infty$:
\[
	\#\mathfrak A_n\cdot\|f-f_n\|_{L_2[0,1]}=o(1).
\]
Then the monotone function $f$ is purely singular.
\end{proposition}

\begin{proposition}\label{5.2}
\textbf{\textsc{(\cite[Proposition~5.2.1]{VSh3})}}
Let $\{\lambda_n\}_{n=0}^{\infty}$ be a sequence of the eigenvalues of boundary value problem
\begin{gather*}
	-y''-\lambda\mu y=0,\\ 
	y'(a)=y'(b)=0,
\end{gather*}
numbered in ascending order. 
Let $\{\nu_n\}_{n=0}^{\infty}$ be a similar sequence corresponding to the boundary value problem
\[
	y'(a)-\gamma_0y(a)=y'(b)+\gamma_1y(b)=0
\]
for the same equation. Then
\[
	\sum\limits_{n=1}^{\infty} |\ln\nu_n-\ln\lambda_n|<+\infty.
\]
\end{proposition}

\begin{remark}
This result could be rewritten as
\[
\prod\limits_{n=1}^\infty \dfrac{\nu_n}{\lambda_n} = \exp\sum\limits_{n=1}^{\infty} |\ln\nu_n-\ln\lambda_n| < +\infty.
\]
For more general results about similar products of eigenvalue ratios see \cite{Naz2}.
\end{remark}

Let the assumptions of Theorem \ref{4.1} be fulfilled. Define $F$ by the relation \eqref{Fdef}. Denote by $\{ \mu_n(J) \}_{n=0}^\infty$ the elements of the collection $\{\lda_n(J_1) \}_{n=0}^\infty \cup \{\lda_n(J_2) \}_{n=0}^\infty$ numbered in ascending order. By Theorem \ref{4.1} we have
\[
F(\lambda) = -1 \quad \Longleftrightarrow \quad \lambda \in \bigcup\limits_{n=0}^{\infty} (\mu_n(J), \lambda_n(J)]. 
\]

We recall that $\mu_0(J) = \lambda_0(J) = \mu_1(J) = 0$, but the rest of $\mu_n(J)$, $\lambda_n(J)$ are greater than zero, and we will now prove, that the set $\{\ln\lambda : F(\lambda) = -1  \}$ has finite measure, i.e. 
\[
\left|\bigcup\limits_{n=2}^{\infty} (\ln\mu_n(J), \ln\lambda_n(J)]\right| < +\infty.
\]

\begin{theorem}\label{Th5.1}
Let the assumptions of Theorem \ref{4.1} be fulfilled and let $c_2-d_1>0$.  Then
\[
\sum\limits_{n=2}^{\infty}|\ln\lambda_{n}(J)-\ln\mu_{n}(J)| < +\infty.
\]
\end{theorem}

\begin{proof}
Denote by $\nu_n^{(1)}$ the eigenvalues of the problem
\begin{gather*}
\left\{
\begin{split}
&-y'' = \nu\mu y, \\
&y'(c_1) = y'(d_1) + \dfrac{2}{c_2-d_1}\cdot y(d_1) = 0,
\end{split}
\right.
\end{gather*}
and by $\nu_n^{(2)}$ --- the eigenvalues of the problem
\begin{gather*}
\left\{
\begin{split}
&-y'' = \nu\mu y, \\
&y'(c_2) - \dfrac{2}{c_2-d_1}\cdot y(c_2) = y'(d_2) = 0.
\end{split}
\right.
\end{gather*}
Let's fix an eigenfunction $y_n$ corresponding to $\lda_n(J)$ and consider its restrictions on segments $J_1$ and $J_2$.
Since $\mu|_{[d_1,c_2]}\equiv 0$, the function $y_n|_{[d_1, c_2]}$ is linear, which means that
\[
\dfrac{y_n(c_2)}{y_n'(c_2)}-\dfrac{y_n(d_1)}{y_n'(d_1)} = c_2-d_1.
\]
This implies
\[
\dfrac{y_n(c_2)}{y_n'(c_2)} \geqslant \dfrac{c_2-d_1}{2} \mbox{ or }
-\dfrac{y_n(d_1)}{y_n'(d_1)} \geqslant \dfrac{c_2-d_1}{2},
\]
which means that one of the following estimates holds:
\begin{equation}\label{eq:5.1}
0 \leqslant -\dfrac{y_n'(d_1)}{y_n(d_1)}\leqslant \dfrac{2}{c_2-d_1},
\end{equation}
or
\[
0 \leqslant \dfrac{y_n'(c_2)}{y_n(c_2)}\leqslant \dfrac{2}{c_2-d_1}.
\]
Note, that $\lambda_n(J)$ is an eigenvalue of the problem 
\begin{gather*}
\left\{
\begin{split}
&-y'' = \lambda\mu y, \\
&y'(c_1) = y'(d_1) + \gamma\cdot y(d_1) = 0,
\end{split}
\right.
\end{gather*}
with $\gamma = -\dfrac{y_n'(d_1)}{y_n(d_1)}$. Its number is the same as the number of zeroes of $y_n$ inside $J_1$. Note also, that $\lambda_k(J_1)$ is the eigenvalue of the same problem with $\gamma=0$ for any $k\in\mathbb{N}$, and $\nu_k^{(1)}$ is an eigenvalue of the same problem with $\gamma =  \dfrac{2}{c_2-d_1}$.

By the variational principle, this implies, that if \eqref{eq:5.1} holds, then 
$$\lda_k(J_1) \leqslant \lda_n(J) \leqslant \nu_k^{(1)},$$ 
where $k$ is the number of zeroes of $y_n$ inside $J_1$. Otherwise, similar argument shows, that $$\lda_k(J_2) \leqslant \lda_n(J) \leqslant \nu_k^{(2)},$$ where $k$ is the number of zeroes of $y_n$ inside $J_2$. Further, by Theorem \ref{4.1} the collections $\{\mu_n(J)\}_{n=0}^\infty$ and $\{\lda_n(J)\}_{n=0}^\infty$ non-strictly interlace starting with $\mu_0(J)$. Therefore, we have $\mu_n(J) \leqslant \lda_n(J) \leqslant \mu_{n+1}(J)$. 
Since all $\lda_k(J_{1,2})$ belong to the set $\{\mu_n(J)\}_{n=0}^\infty$, the relation $\lda_k(J_{1,2})\leqslant\lda_n(J)$ implies $\lda_k(J_{1,2})\leqslant\mu_n(J)$. 
Thus we obtain for every $n$, that if \eqref{eq:5.1} holds, then there exists $k$, such that
$$\lda_k(J_1) \leqslant \mu_n(J) \leqslant \lda_n(J) \leqslant \nu_k^{(1)},$$ 
otherwise, there exists $k$, such that
$$\lda_k(J_2) \leqslant \mu_n(J) \leqslant \lda_n(J) \leqslant \nu_k^{(2)}.$$ 
Thus, each of the non-intersecting intervals $ (\mu_n(J), \lda_n(J)]$ is contained in the union $\Big( \bigcup\limits_{k=0}^\infty [\lda_k(J_1), \nu_k^{(1)}] \Big) \cup \Big( \bigcup\limits_{k=0}^\infty [\lda_k(J_2), \nu_k^{(2)}] \Big)$, which implies
\begin{equation}\label{th4eq0}
\bigcup\limits_{n=0}^\infty (\mu_n(J), \lda_n(J)] \subset \Big( \bigcup\limits_{k=0}^\infty [\lda_k(J_1), \nu_k^{(1)}] \Big) \cup \Big( \bigcup\limits_{k=0}^\infty [\lda_k(J_2), \nu_k^{(2)}] \Big),
\end{equation}
and if we discard the segments corresponding to $k=0$ in the right part of \eqref{th4eq0}, then we will only need to discard a finite number of intervals in the left part, in other words, there exists a number $n_0$, such that 
\begin{equation*} 
\bigcup\limits_{n=n_0}^\infty (\mu_n(J), \lda_n(J)] \subset \Big( \bigcup\limits_{k=1}^\infty [\lda_k(J_1), \nu_k^{(1)}] \Big) \cup \Big( \bigcup\limits_{k=1}^\infty [\lda_k(J_2), \nu_k^{(2)}] \Big),
\end{equation*}
and thus
\begin{gather}\label{th4eq1}
\begin{split}
\sum\limits_{n=n_0}^{\infty}|\ln\lambda_{n}(J)-\ln\mu_{n}(J)| &\leqslant 
 \sum\limits_{k=1}^{\infty}|\ln\nu_{k}^{(1)}-\ln\lda_{k}(J_1)| \\
 {} &+ \sum\limits_{k=1}^{\infty}|\ln\nu_{k}^{(2)}-\ln\lda_{k}(J_2)|.
\end{split}
\end{gather} 
From \eqref{th4eq1} using Proposition \ref{5.2} 
we obtain the estimate
\begin{gather*}
\sum\limits_{n=2}^{\infty}|\ln\lambda_{n}(J)-\ln\mu_{n}(J)| \leqslant 
 \sum\limits_{n=2}^{n_0-1}|\ln\lambda_{n}(J)-\ln\mu_{n}(J)| \\
 {} + \sum\limits_{k=1}^{\infty}|\ln\nu_{k}^{(1)}-\ln\lda_{k}(J_1)| +
 \sum\limits_{k=1}^{\infty}|\ln\nu_{k}^{(2)}-\ln\lda_{k}(J_2)| < +\infty.
\end{gather*}
\end{proof}

\begin{proof}[Proof of Theorem \ref{main_theo}]
By \eqref{eq:mes_asymp} we have
\[
N(\lda) = \lda^D\big(s(\ln \lda)+\varepsilon(\lda)\big),
\]
where $\varepsilon(\lda)\to 0$ as $\lda\to\infty$. For arbitrary $k\in\mathbb{N}$ we obtain
\[
N(\tau^{-k}\lda) = \tau^{-kD} \lda^D(s(\ln\lda) + \varepsilon(\tau^{-k}\lda)),
\]
whence $$s(\ln\lda)\lda^D = \lim\limits_{k\to\infty} \tau^{kD}N(\tau^{-k}\lda)$$
uniformly on any segment. Denote
\[
\sigma(t) := s(t) e^{Dt} = \lim\limits_{k\to\infty} \tau^{kD}N(\tau^{-k} e^t),
\quad t\in[0,T].
\]
We need to define a more suitable sequence of step functions approximating $\sigma$.

\paragraph{The case of $m=2$.}
Let's first consider for clarity the case of $m=2$.
Without loss of generality, consider $k_1\leqslant k_2$, where $k_{1,2}$ are introduced in Definition \ref{arithm}. Denote
\[
f_j(t) := C\tau^{jD}\sum_{i=0}^{k_2-1} C_i N(\tau^{-i-j}e^t),
\quad t\in[0,T].
\]
We wish to choose the coefficients $C_i$ in such a way, that
\begin{equation}\label{eq:5.2}
f_{j+1}(t) - f_j(t) = C\tau^{jD}\Big(N(\lda) - N(\tau^{k_1}\lda) - N(\tau^{k_2}\lda)\Big),
\end{equation}
where $\lda = \tau^{-k_2-j}e^t$. By direct calculations using \eqref{period} it is easy to demonstrate, that the wanted coefficient values are
\[
C_i = \left\{
\begin{aligned}
&\tau^{-(k_2-i)D} \quad && k_2-i=1,\ldots, k_1, \\
&\tau^{-(k_2-i)D}\cdot\big(1-\tau^{k_1 D}\big) \quad && k_2-i=k_1+1,\ldots, k_2.
\end{aligned}
\right.
\]
The second line doesn't manifest when $k_1 = k_2$.
Note, that
\[
\lim_{j\to+\infty}f_j(t) = C\sum\limits_{i=0}^{k_2-1} C_i\tau^{-iD} \lim_{j\to+\infty}\tau^{(i+j)D}N(\tau^{-i-j}e^t) = \sigma(t) \cdot C\sum\limits_{i=0}^{k_2-1} C_i\tau^{-iD},
\]
so, if we assign
\[
C := \Big(\sum_{i=0}^{k_2-1} C_i \tau^{-iD}\Big)^{-1},
\]
then we have
\[
\sigma(t) = \lim_{j\to+\infty}f_j(t).
\]
Note, also, that by Lemma \ref{lemma1} we have
\begin{equation}\label{eq:5.3}
N(\lda) - N(\tau^{k_1}\lda) - N(\tau^{k_2}\lda) = N(\lda, [0,1]) - N(\lda, I_1) - N(\lda, I_2),
\end{equation}
thus, using Theorem \ref{4.1} we obtain from \eqref{eq:5.2} the following estimate:
\begin{equation}\label{fmod}
|f_{j+1}(t) - f_j(t)| \leq C\tau^{jD}.
\end{equation}
Further, since $f_j$ is a sum of $k_2$ terms, each having no more than $N(\tau^{-k_2-j})$ discontinuity points, the number of discontinuities of $f_j$ could be estimated as
\begin{equation}\label{fbreaks}
\#\mathfrak A_j \leq k_2 N(\tau^{-k_2-j}) \leq \widetilde{C} \tau^{-jD}.
\end{equation}
All that's left in order to use the Proposition \ref{sing} is to prove the estimate
\begin{equation}\label{fmeas}
mes\{t\in[0, T]: f_j(t)\neq f_{j+1}(t)\} = o(1), \quad j\to\infty.
\end{equation}
Denote by $\{\mu_n\}_{n=0}^\infty$ the elements of the collection $\{\lda_n(I_1) \} \cup \{\lda_n(I_2) \}$ numbered in ascending order. 
Since $\lda_n$ and $\mu_n$ non-strictly interlace starting with $\mu_0$, we have $\mu_n \leqslant \lda_n$ for all $n\geqslant 0$.  Considering \eqref{eq:5.2} and \eqref{eq:5.3}, in order to prove \eqref{fmeas}, we need to estimate the Lebesgue measure of the set of $t$, for which
\[
N(\lda, [0,1]) - N(\lda, I_1) - N(\lda, I_2) \neq 0,
\]
where $\lda = \tau^{-k_2-j}e^t$, $t\in[0, T]$.
That's only true, when
\[
\lda \in \bigcup\limits_{n=0}^\infty 
\left[\mu_{n}, \lda_{n}\right],
\]
and considering $\lda = \tau^{-k_2-j}e^t$ and $t\in[0, T]$, we obtain
\[
(k_2+j)T + t \in \Big( 
\bigcup\limits_{n=0}^\infty 
\left[\ln\mu_{n}, \ln\lda_{n}\right]
\Big) \cap [(k_2+j)T, (k_2+j+1)T].
\]
The measure of the union
\[
\Big| 
\bigcup\limits_{n=2}^\infty 
\left[\ln\mu_{n}, \ln\lda_{n}\right]
\Big| = \sum\limits_{n=2}^{\infty}|\ln\lambda_{n}-\ln\mu_{n}|
\]
is bounded by Theorem \ref{Th5.1}, which means, that the measure of its intersection with segments, that are moving to infinity, tends to zero, which proves the estimate \eqref{fmeas}.

From \eqref{fmod} and \eqref{fmeas} we obtain
\[
\|f_{j+1} - f_j\|_{L_2[0,T]} = o(\tau^{jD}),
\]
whence
\[
\|\sigma - f_j\|_{L_2[0,T]} = o(\tau^{jD}).
\]
Using this estimate and \eqref{fbreaks}, we obtain
\[
\#\mathfrak A_n\cdot\|\sigma-f_n\|_{L_2[0,T]} = o(1),
\]
which allows us to use Proposition \ref{sing} for the function $\sigma$, concluding the proof of the theorem for this case.

\paragraph{General case.}
Let $\{\kappa_i\}_{i=1}^{p}$ be the elements of the set $\{k_i\}_{i=1}^m$ numbered in ascending order without duplication, $\{l_i\}_{i=1}^{p}$ --- their multiplicities (the number of segments $I_n$ corresponding to each value). Similarly to the case of $m=2$ we define functions and constants 
\[
f_j(t) := C\tau^{jD}\sum_{i=0}^{\kappa_p-1} C_i N(\tau^{-i-j}e^t),
\quad t\in[0,T],
\]
\[
C_i = 
\left\{
\begin{aligned}
&\tau^{-(\kappa_p-i)D} \quad  && \kappa_p-i=1,\ldots, \kappa_1, \\
&\tau^{-(\kappa_p-i)D}\cdot\big(1-l_1 \tau^{\kappa_1 D}\big) \quad  && \kappa_p-i=\kappa_1+1,\ldots, \kappa_2, \\
&\tau^{-(\kappa_p-i)D}\cdot\big(1-l_1 \tau^{\kappa_1 D}-l_2 \tau^{\kappa_2 D}\big) \quad  && \kappa_p-i=\kappa_2+1,\ldots, \kappa_3, \\
&\ldots && \\
&\tau^{-(\kappa_p-i)D}\cdot\Big(1-\sum_{j=1}^{p-1} l_j \tau^{\kappa_j D}\Big) \quad  && \kappa_p-i=\kappa_{p-1}+1,\ldots, \kappa_p,
\end{aligned}
\right.
\]
\[
C = \Big(\sum_{i=0}^{\kappa_p-1} C_i \tau^{-iD}\Big)^{-1}.
\]
From \eqref{period} we have
\[
\sum_{i=1}^{p} l_i\tau^{\kappa_i D} = 1,
\]
whence
\[
1 - \sum_{i=1}^{r} l_i\tau^{\kappa_i D} > 0, \quad r=1,\ldots, p-1.
\]
This means, that coefficients $C_i$ are positive, and $C$ is well defined. Also, similarly to the previous case, such a choice of coefficients $C_i$ gives us the relation
\[
f_{j+1}(t) - f_j(t) = C\tau^{jD}\Big(N(\lda) - \sum_{i=1}^p l_i N(\tau^{\kappa_i}\lda)\Big),
\]
where $\lda = \tau^{-\kappa_p-j}e^t$, and the choice of $C$ gives us the relation
\[
\sigma(t) = \lim_{j\to+\infty}f_j(t).
\]
Using Lemma \ref{lemma1}, we obtain
\[
f_{j+1}(t) - f_j(t) = 
C\tau^{jD}\Big(N(\lda, [0,1]) - \sum_{i=1}^m N(\lda, I_i)\Big)
\]
\[
{} = C\tau^{jD}\sum_{i=1}^{m-1}\Big( N(\lda, [a_i, 1]) - N(\lda, I_i) - N(\lda, [a_{i+1}, 1])\Big).
\]
Considering Theorem \ref{4.1}, it is easy to see, that
\[
|f_{j+1}(t) - f_j(t)| \leq C(m-1)\tau^{jD},
\]
and the number of discontinuities of $f_j$ could be estimated as
\[
\#\mathfrak A_j \leq \kappa_p N(\tau^{-\kappa_p-j}) \leq \widetilde{C} \tau^{-jD}.
\]
All that's left to use Proposition \ref{sing} is to prove the estimate
\[
mes\{t\in[0, T]: f_j(t)\neq f_{j+1}(t)\} = o(1), \quad j\to\infty,
\]
which follows directly from Theorem \ref{Th5.1}, same as in the case of $m=2$.
\end{proof}

\section{Small ball deviations of Gaussian processes}
The problem of small ball deviations of Gaussian processes was studied intensively in the recent decades (see, e.g., the reviews \cite{Lif,LiSh,Fatalov}).
Here we recall the statement of the problem.

Consider a Gaussian process $X(t)$, $t \in [0,1]$, with zero mean, and denote by $G_X(t, s) = EX(t)X(s)$, $t, s \in [0,1]$ its covariance function.
Let $\mu$ be a measure on $[0,1]$. Denote
\[
\|X\|_\mu = \Big( \int_{0}^1  X^2(t) \mu(dt)\Big)^{1/2}.
\]
We call an asymptotics of \mbox{$\ln{\bf P}\{\|X\|_\mu \leq \varepsilon\}$} as $\varepsilon\to 0$ a \emph{logarithmic small ball asymptotics}.

It is possible to connect the study of the small ball asymptotics to the study of the eigenvalues of an integral operator
\begin{equation}\label{eq:int_eq}
\uplambda y(t) = \int\limits_0^1 G_X(s,t) y(s) \mu(ds), \quad t\in[0,1].
\end{equation}
If $G_X$ is the Green function for a boundary value problem for the Sturm--Liouville equation, then the eigenvalues of \eqref{eq:int_eq} and the eigenvalues of the Sturm--Liouville boundary problem are the inverse values of each other. Therefore, the asymptotics \eqref{eq:mes_asymp} considering Remark \ref{remark1} gives us the asymptotics
\begin{equation}\label{uplambda_asymp}
\uplambda_n = \dfrac{\varphi(\ln n)}{n^{1/D}} (1 + o(1)), \qquad n\to+\infty,
\end{equation}
where periodic function $\varphi$ satisfies
\begin{equation}\label{eq_varphi_s1}
\varphi(x) = \left(s\left(x/D + \ln(\varphi(x))\right)\right)^{1/D}.
\end{equation}
\begin{remark}
The class of Green Gaussian processes with such covariation functions includes, among others, the Wiener process, the Brownian bridge, the generalized Slepian process (see \cite{Slepian}, \cite{Shepp}) and the Ornstein--Uhlenbeck process.
\end{remark}

\begin{proposition}
\textbf{\textsc{(\cite[Proposition 5.1]{Naz})}}
Let $X(t)$ be a Green Gaussian process corresponding to a second order differential operator. 
Let $\mu$ be an arithmetically self-similar singular measure. Then
\begin{equation}\label{sb_log_asymp}
\ln{\bf P}\{\|X\|_\mu \leq \varepsilon\} \sim -\varepsilon^{-\frac{2D}{1-D}} \zeta(\ln(1/\varepsilon)),
\end{equation}
where $D$ is the power exponent of the asymptotics \eqref{eq:mes_asymp} for measure $\mu$, $\zeta$ is a bounded and separated from zero $\frac{T(1-D)}{2}$-periodic function, $T$ is the period of the function $s$ from the asymptotics \eqref{eq:mes_asymp}.
\end{proposition}

The main focus of this work is on proving the periodic component of spectral asymptotics does not degenerate into constant. Here we prove the same about the periodic component of the small ball asymptotics in \eqref{sb_log_asymp}.

\begin{theorem}\label{th_inconst_zeta}
For a fixed measure $\mu$, the component $\zeta$ of the small ball asymptotics \eqref{sb_log_asymp}  degenerates into constant iff the component $s$ of the spectral asymptotics \eqref{eq:mes_asymp} degenerates into constant. 
\end{theorem}

\begin{proof}
It follows from the proof of \cite[Theorem 4.2]{Naz}, that
\begin{equation}\label{lnP_asymp}
\ln{\bf P}\{\|X\|^2_\mu \leq r\} \sim - (u(r))^D \cdot \eta (\ln (u(r))), \qquad r\to 0,
\end{equation}
where $u(r)$ is an arbitrary function satisfying
\begin{equation}\label{u_def2}
r \sim u^{D-1} \theta(\ln(u)), \quad u \to \infty.
\end{equation}
Periodic functions $\eta$ and $\theta$ in \eqref{lnP_asymp}--\eqref{u_def2} are defined by formulae
\begin{equation}\label{eq_eta_phi}
\eta(\ln(u)) = \int\limits_0^\infty F\left(\dfrac{\varphi(D\ln u + \ln z)}{z^{1/D}}\right)\, dz,
\end{equation}
\begin{equation}\label{eq_theta_phi}
\theta(\ln(u)) = \int\limits_0^\infty F_1\left(\dfrac{\varphi(D\ln u + \ln z)}{z^{1/D}}\right)\, dz,
\end{equation}
where
\[
F(x) = \dfrac{1}{2} \ln(1+2x) - \dfrac{x}{1+2x}, \quad 
F_1(x) = \dfrac{x}{1+2x}.
\]
and $\varphi$ is a periodic function defined in \eqref{eq_varphi_s1}.

To reiterate, we now have periodic functions $\zeta$, $\eta$, $\theta$, $\varphi$ and $s$ interconnected via various relations, and we aim to show an equivalency between their conditions for degenerating into constant. Let's break it down into simpler steps.

\paragraph{Step 1.} Function $\zeta$ is constant iff function $\eta\, \theta^{\frac{D}{1-D}}$ is constant.

If we substitute $u$ from \eqref{u_def2} into \eqref{lnP_asymp} we obtain
\[
\ln{\bf P}\{\|X\|^2_\mu \leq r\} \sim - r^{-\frac{D}{1-D}} \cdot \eta (\ln (u(r)))\,\theta^{\frac{D}{1-D}} (\ln (u(r))), \qquad r\to 0.
\]
Thus, function $\zeta$ in \eqref{sb_log_asymp} has the asymptotics
\begin{equation}\label{eq_zeta_phi}
\zeta(\ln(1/\varepsilon)) \sim \eta(\ln (u(\varepsilon^2)))\,\theta^{\frac{D}{1-D}} (\ln (u(\varepsilon^2))), \quad \varepsilon \to 0. 
\end{equation} 
Since $u$ is a monotone function taking all values in some vicinity of $\infty$, we conclude that function $\zeta$ is constant iff function $\eta\, \theta^{\frac{D}{1-D}}$ is constant.

\paragraph{Step 2.} Function $\eta\, \theta^{\frac{D}{1-D}}$ is constant iff both $\theta$ and $\eta$ are constant. 

It is clear, that if $\theta$ and $\eta$ are constant, then the product $\eta\, \theta^{\frac{D}{1-D}}$ is constant as well. To prove the other implication, we change the variable $x = u^D z$ in \eqref{eq_eta_phi} and \eqref{eq_theta_phi} to obtain 
\begin{equation}\label{eta_theta_new_def}
\eta(\ln u) u^D = \int\limits_0^\infty F(u \psi(x))\, dx, \qquad \theta(\ln u) u^D = \int\limits_0^\infty F_1(u \psi(x))\, dx,
\end{equation}
where $\psi(x) = \varphi(\ln(x))/x^{1/D}$. Differentiating these formulae and taking into account the relation
\[
(F(x) + F_1(x))'\cdot x = F_1(x),
\]
we arrive at
\[
(\eta(\ln u ) u^D + \theta(\ln u) u^D)' \cdot u = \theta(\ln u) u^D,
\]
thus
\[
 \eta'(x) + D\eta(x) + \theta'(x) - (1-D)\theta(x) = 0.
\]
If $\eta\, \theta^{\frac{D}{1-D}} \equiv c$ here, then minimal and maximal values of $\theta$ on the period satisfy
\[
cD \theta^{-\frac{D}{1-D}} - (1 - D)\theta = 0
\]
therefore they coincide and the claim follows.

\paragraph{Step 3.} If function $\varphi$ is constant then both functions $\eta$ and $\theta$ are constant.

This follows trivially from their definitions \eqref{eq_eta_phi} and \eqref{eq_theta_phi}.

\paragraph{Step 4.} If function $\theta$ is constant then function $\varphi$ is constant.

This statement is proved in Lemma \ref{lemma_theta_phi} below.

\paragraph{Step 5.} Function $\varphi$ is constant iff function $s$ is constant.

Function $\varphi$ is defined by the relation \eqref{eq_varphi_s1}, 
thus $\varphi$ and $s$ are constant at the same time. 

\medskip

To conclude the proof of the theorem we unify the steps above:
\begin{align*}
\zeta=const & \xLeftrightarrow{\text{Step 1}} \eta\, \theta^{\frac{D}{1-D}}=const \xLeftrightarrow{\text{Step 2}} \eta=const \text{ and } \theta=const \\ 
{} & \xLeftrightarrow{\text{Steps 3,4}} \varphi=const \xLeftrightarrow{\text{Step 5}} s=const.
\end{align*}
\end{proof}

\begin{lemma}\label{lemma_theta_phi}
Let $\theta$ be a $T$-periodic function defined by \eqref{eq_theta_phi} with $TD$-periodic function $\varphi$. If $\theta$ is constant, then function $\varphi$ is also constant.
\end{lemma}

\begin{proof}
Let $\theta \equiv c$ be constant. 
From \eqref{eta_theta_new_def} we obtain
\begin{equation}\label{integral_psi}
c u^D = \int\limits_0^\infty F_1(u \psi(x))\, dx.
\end{equation}
We note the following properties of the function $\psi(x) = \varphi(\ln(x))/x^{1/D}$:
\begin{enumerate}
    \item From the periodicity of $\varphi$ we obtain
    \begin{equation}\label{periodicity_psi}
        \psi(x e^{-DT}) = \psi(x) e^{T}.
    \end{equation}
    \item We claim, that $\psi$ is a monotone decreasing function. Indeed, if it is not then there exist $x_1 > x_2$, such that $\psi(x_1) > \psi(x_2)$. However, for any $k\in\mathbb{N}$ formulae \eqref{uplambda_asymp} and \eqref{periodicity_psi} imply 
    \[
    e^{-kT} (\psi(x_1) - \psi(x_2)) = \psi(x_1 e^{kDT}) - \psi(x_2 e^{kDT}) = \uplambda_{n_1} - \uplambda_{n_2} + o(e^{-kT}),
    \]
    where $n_{1,2} = \lfloor x_{1,2}e^{kDT} \rfloor$ are the integer parts. Thus, for sufficiently large $k$ we will be able to find two eigenvalues $\uplambda_{n_1} > \uplambda_{n_2}$, $n_1 > n_2$, which is impossible.
\end{enumerate}
On the other hand, direct computation gives
\begin{equation}\label{integral_phi}
\int\limits_0^\infty F_1(u \phi(x))\, dx = c u^D,
\end{equation}
where
\[
\phi(x) = c_1/x^{1/D}, \qquad c_1 = \left[ \dfrac{1}{c}  \int\limits_0^\infty F_1(x^{-1/D})\, dx\right]^{-1/D}.
\] 
We change variables in integral identities \eqref{integral_psi} and \eqref{integral_phi} to $w = - \ln (u \psi(x))$ and $w = - \ln( u \phi(x))$ respectively, subtract the resulting relations and arrive at
\begin{equation}\label{conv1}
\forall v \in \mathbb{R} \qquad \int\limits_{-\infty}^{+\infty} F_1(e^{-w}) \left[ \dfrac{1}{\psi'(\psi^{-1}(e^{v-w}))} - \dfrac{1}{\phi'(\phi^{-1}(e^{v-w}))} \right] \, e^{v-w} dw = 0,
\end{equation}
where  $v = -\ln u$.

The integral in \eqref{conv1} is in essence a convolution. Namely, we multiply  \eqref{conv1} by $e^{Dv}$ 
and rewrite it as
\begin{equation}\label{conv2}
\forall v \in \mathbb{R} \qquad (g_1 * g_2) (v) = 0,
\end{equation}
where
\[
g_1(v) = F_1(e^{-v}) e^{Dv}, \qquad g_2(v) = \left[ \dfrac{1}{\psi'(\psi^{-1}(e^{v}))} - \dfrac{1}{\phi'(\phi^{-1}(e^{v}))} \right] \, e^{(1+D)v}.
\]
 We aim to demonstrate that $g_2 \equiv 0$ using the Fourier transform argument.

\paragraph{Step 1.} We show that $g_1$ and $g_2$ have valid Fourier transforms. By the definition of $F_1$ it is easy to see that
\[
    g_1(v) \sim e^{(D-1)v}, \quad v\to +\infty, \qquad g_1(v) \sim \frac{1}{2} e^{Dv}, \quad v\to -\infty,
\]
thus $g_1 \in L_1(\mathbb{R})$, and its Fourier transform is a continuous function on $\mathbb{R}$. 

Further, we show that $g_2$ is a $T$-periodic function. Note, that by the definition of $\phi$ we have
\begin{equation}\label{phi_deriv_of_phi_inverse}
\phi'(\phi^{-1}(1/u)) = -\dfrac{u^{-D-1}}{Dc_1^D}.   
\end{equation} 
Therefore the second term of $g_2$ is a constant:
\[
 - \dfrac{e^{(1+D)v}}{\phi'(\phi^{-1}(e^{v}))} = D c_1^D.
\]
For the first term of $g_2$ we obtain
\begin{align*}
\forall a \in \mathbb{R} \quad \int\limits_a^{a+T} \dfrac{e^{(1+D)v}}{\psi'(\psi^{-1}(e^{v}))} dv &\stackrel{(a)}{=} \int\limits_{\psi^{-1}(e^a)}^{\psi^{-1}(e^a e^{T})} (\psi(z))^D dz \stackrel{(b)}{=} \int\limits_{\psi^{-1}(e^a)}^{\psi^{-1}(e^a) e^{-TD}} \dfrac{(\varphi(\ln z))^D}{z} dz \\
{} &\stackrel{(c)}{=} \int\limits_{\ln \psi^{-1}(e^a)}^{ \ln \psi^{-1}(e^a) - TD} \varphi(\xi) d\xi = const.
\end{align*}
The equality $(a)$ is obtained by the change of variable $z = \psi^{-1}(e^v)$. The equality $(b)$ uses the property \eqref{periodicity_psi} for the upper limit and the definition of $\psi$ for the integrand. Finally, in $(c)$ we change variable $\xi = \ln z$ and obtain an integral of a periodic function $\varphi$ over its period $TD$. Thus $g_2$ is a $T$-periodic functions and its Fourier transform (in the sense of distributions) is a linear combination of delta-functions at points $\frac{kT}{2\pi}$ for $k\in\mathbb{Z}$.

\paragraph{Step 2.} We show that the Fourier transform of $g_1$ is never zero:
\[
\forall \omega\in \mathbb{R} \qquad \widehat{g}_1 (\omega) = \int\limits_{-\infty}^{+\infty} F_1(e^{-v})e^{Dv -iv\omega} dv \neq 0.
\]
For $\omega > 0$ we calculate this integral using the residue theorem for a semicircular contour in the lower imaginary half-plane. Singularities of function $F_1(e^{-v})$ are simple poles at points $\xi_k = \ln 2 + i \pi (2k-1)$ for all $k\in\mathbb{Z}$. We calculate the residues using the L'H\^{o}pital rule
\[
\mathrm{Res} \left(\dfrac{f}{h}, a \right) = 
\dfrac{f(a)}{h'(a)}, 
\]
if $f(a) \neq 0$, $h(a) = 0$, $h'(a) \neq 0$. In our case 
\[
\mathrm{Res} (F_1(e^{-v})e^{v(D - i\omega)}, \xi_k) = -\dfrac{1}{2} e^{\xi_k (D-i\omega)} = {}
\]
\[
{} = -\dfrac{1}{2} e^{(D - i\omega)(\ln 2 - i\pi)} (e^{2\pi i(D - i\omega)})^k.
\]
Summing all residues in the lower half-plane we obtain 
\[
\int\limits_{-\infty}^{+\infty} F_1(e^{-v})e^{v(D - 1 -i\omega)} dv = 2\pi i \sum_{k=0}^{\infty} \dfrac{1}{2}e^{(D - i\omega)(\ln 2 - i\pi)} (e^{2\pi i (D - i\omega)})^{-k}
\]
\[
{} = \dfrac{\pi i e^{(D - i\omega)(\ln 2 - i\pi)}}{1 - e^{-2\pi i(D - i\omega)}} \neq 0.
\]
For $\omega < 0$
\[
\widehat{g}_1 (\omega) = \overline{\widehat{g}_1 (-\omega)} = \dfrac{-\pi i e^{(D - i\omega)(\ln 2 + i\pi)}}{1 - e^{2\pi i(D - i\omega)}} \neq 0.
\]
Finally,
\[
\widehat{g}_1 (0) = \dfrac{-\pi i e^{D(\ln 2 + i\pi)}}{1 - e^{2\pi i D}} = \dfrac{\pi 2^{D-1}}{ \sin (\pi D)} \neq 0.
\]

\paragraph{Step 3.} Now we are ready to apply the Fourier transform to the relation \eqref{conv2}:
\[
\forall \omega \in \mathbb{R} \qquad \widehat{g}_1(\omega) \cdot \widehat{g}_2(\omega) = 0.
\]
And since $\widehat{g}_1(\omega)\neq 0$ we conclude that $\widehat{g}_2 \equiv 0$ and thus 
\[
\forall u > 0 \qquad u^{-1-D} g_2(\ln u) = \dfrac{1}{\psi'(\psi^{-1}(u))} - \dfrac{1}{\phi'(\phi^{-1}(u))} = 0.
\]
By \eqref{phi_deriv_of_phi_inverse} it immediately follows that
\[
\forall x > 0 \qquad \psi'(x) = -\dfrac{\psi^{D+1}(x)}{D c_1^D} .
\]
This differential equation has a family of solutions
\[
\psi(x) = \dfrac{c_1}{(x+c_2)^{1/D}},
\]
and since $\psi(x) x^{1/D} = \varphi(\ln(x))$ is a periodic function of logarithm, we obtain $\varphi \equiv c_1$, and the lemma is proved.
\end{proof}

\begin{remark}
It is evident from \eqref{eta_theta_new_def} that $\eta$ and $\theta$ are smooth functions. Thus, function $\zeta$ in \eqref{sb_log_asymp} is smooth even though function $s$ in \eqref{eq:mes_asymp} is not.
\end{remark}

\bigskip

Author thanks A.~I.~Nazarov for the statement of the problem and constant encouragement, D.~D.~Cherkashin for valuable remarks, and A.~A.~Vladimirov for the remark, that allowed to simplify the proof of Theorem \ref{4.1} greatly.

The work is supported by Ministry of Science and Higher Education of the Russian Federation (agreement no. 075-15-2022-287).
\bigskip
\bigskip

\begin{enbibliography}{99}
\addcontentsline{toc}{section}{References}

\bibitem{VSh3} 
Vladimirov, A.~A. and Sheipak, I.~A., 2013. On the Neumann problem for the Sturm--Liouville equation with Cantor-type self-similar weight. Funktsional'nyi Analiz i Ego Prilozheniya, 47(4), pp.~18-29 (in Russian). English translation in Functional Analysis and Its Applications, 47(4), 2013, pp.~261-270.

\bibitem{Rast1} 
Rastegaev, N.~V., 2014. On spectral asymptotics of the Neumann problem for the Sturm--Liouville equation with self-similar weight of generalized Cantor type. Zapiski Nauchnykh Seminarov POMI, 425, pp.~86-98 (in Russian). English translation in Journal of Mathematical Sciences, 210, 2015, pp.~814-821.

\bibitem{BS2} Birman, M.~Sh. and Solomyak, M.~Z., 2010. Spectral theory of self-adjoint operators in Hilbert space. 
Ed.2, Lan' publishers (in Russian). English translation of the 1st ed: Mathematics and its Applications (Soviet Series). D. Reidel Publishing Co., Dordrecht, 1987.

\bibitem{Mats} 
Markus, A.~S. and Matsaev, V.~I., 1982. Comparison theorems for spectra of linear operators and spectral asymptotics. Trudy Moskovskogo Matematicheskogo Obshchestva, 45, pp.~133-181 (in Russian).

\bibitem{NazSheip} 
Nazarov, A.~I. and Sheipak, I.~A., 2012. Degenerate self-similar measures, spectral asymptotics and small deviations of Gaussian processes. Bulletin of the London Mathematical Society, 44(1), pp.~12-24.

\bibitem{Naz}
Nazarov, A.~I., 2004. Logarithmic $L_2$-small ball asymptotics with respect to self-similar measure for some Gaussian processes. Zapiski Nauchnykh Seminarov POMI, 311, pp.~190-213 (in Russian). English translation in Journal of Mathematical Sciences (New York), 133(3), 2006, pp.~1314-1327.

\bibitem{K} Krein, M.~G., 1951. Determination of the density of the symmetric inhomogeneous string by spectrum.
Doklady Akademii Nauk SSSR,  76(3), pp.~345-348 (in Russian).

\bibitem{BS1} 
Birman, M.~S. and Solomyak, M.~Z., 1970. Asymptotic behavior of the spectrum of weakly polar integral operators. Izvestiya Rossiiskoi Akademii Nauk. Seriya Matematicheskaya, 34(5), pp.~1142-1158 (in Russian). English translation in Mathematics of the USSR-Izvestiya, 4(5), 1970, pp.~1151-1168.

\bibitem{KrKac} 
Kac, I.~S. and Krein, M.~G., 1958. A discreteness criterion for the spectrum of a singular string. Izvestiya Vuzov. Matematika, 3(2), pp.~136-153 (in Russian).

\bibitem{McKeanRay} McKean, H.~P. and Ray, D.~B., 1962. Spectral distribution of a differential operator. Duke Mathematical Journal, 29(2), pp.~281-292.

\bibitem{B} 
Borzov, V.~V., 1970.  Quantitative characteristics of singular measures. Spectral Theory and Wave Processes, volume 4 of Problems of Mathematical Physics, Leningrad University Publisher, pp.~42-47 (in Russian). English translation in Spectral Theory and Wave Processes. Topics in Mathematical Physics, vol 4, Springer, Boston, 1971, pp.~37-42.

\bibitem{F} 
Fujita, T., 1987. A fractional dimension, self-similarity and a generalized diffusion operator. Probabilistic Method on Mathematical Physics. In Proc. of Taniguchi International Symp., pp.~83-90. Kinokuniya.

\bibitem{HU} 
Uno, T. and Hong, I., 1959. Some consideration of asymptotic distribution of eigenvalues for the equation $d^2u/dx^2 + \lda\rho(x)u = 0$. In Japanese journal of mathematics: transactions and abstracts, 29, pp.~152-164. The Mathematical Society of Japan.

\bibitem{RFr}
Freiberg, U.~R. and Rastegaev, N.~V., 2018. On spectral asymptotics of the Sturm--Liouville problem with self-conformal singular weight with strong bounded distortion property. Boundary-value problems of mathematical physics and related problems of function theory. Part 47, Zapiski Nauchnykh Seminarov POMI, 477, pp.~129-135 (in Russian). English translation in Journal of Mathematical Sciences, 244, 2020, pp.~1010-1014.

\bibitem{SV} 
Solomyak, M. and Verbitsky, E., 1995. On a spectral problem related to self-similar measures. Bulletin of the London Mathematical Society, 27(3), pp.~242-248.

\bibitem{KL} 
Kigami, J. and Lapidus, M.~L., 1993. Weyl's problem for the spectral distribution of Laplacians on pcf self-similar fractals. Communications in Mathematical Physics, 158, pp.~93-125.

\bibitem{VSh1} 
Vladimirov, A.~A. and Sheipak, I.~A., 2006. Self-similar functions in \(L_2[0,1]\) and the Sturm--Liouville problem with singular indefinite weight. Matematicheskii Sbornik, 197(11), pp.~13-30 (in Russian). English translation in Sbornik: Mathematics, 197(11), 2006, pp.~1569-1586.

\bibitem{V3} 
Vladimirov, A.~A., 2015. Method of oscillation and spectral problem for four-order differential operator with self-similar weight. Algebra i Analiz, 27(2), pp.~83-95 (in Russian). English translation in St. Petersburg Mathematical Journal, 27(2), 2016, pp.~237-244.

\bibitem{Sh} 
Sheipak, I.~A., 2007. On the construction and some properties of self-similar functions in the spaces \(L_p[0,1]\). Matematicheskie Zametki, 81(6), pp.~924-938 (in Russian). English translation in Mathematical Notes, 81, 2007, pp.~827-839.

\bibitem{H} 
Hutchinson, J.~E., 1981. Fractals and self similarity. Indiana University Mathematics Journal, 30(5), pp.~713-747.

\bibitem{V2} 
Vladimirov, A.~A., 2009. On the oscillation theory of the Sturm--Liouville problem with singular coefficients. Zhurnal Vychislitel'noi Matematiki i Matematicheskoi Fiziki, 49(9), pp.~1609-1621 (in Rusian). English translation in Computational Mathematics and Mathematical Physics, 49, 2009, pp.~1535-1546.

\bibitem{Naz2} 
Nazarov, A.~I., 2009. On a set of transformations of Gaussian random functions. Teoriya Veroyatnostei i ee Primeneniya, 54(2), pp.~209-225 (in Russian). English translation in Theory of Probability and Its Applications, 54(2), 2010, pp.~203-216.

\bibitem{Lif} 
Lifshits, M.~A., 1999. Asymptotic behavior of small ball probabilities. Probability Theory and Mathematical Statistics: Proceedings of the Seventh International Vilnius Conference, pp.~453-468.

\bibitem{LiSh} 
Li, W.~V. and Shao, Q.~M., 2001. Gaussian processes: inequalities, small ball probabilities and applications. Stochastic Processes: Theory and Methods, Handbook of Statistics, 19, pp.~533-597.

\bibitem{Fatalov}
Fatalov, V.~R., 2003. Constants in the asymptotics of small deviation probabilities for Gaussian processes and fields. Uspekhi Matematicheskikh Nauk, 58(4)(352), pp.~89-134 (in Russian). English translation in Russian Mathematical Surveys, 58(4), 2013, pp.~725-772.

\bibitem{Slepian} 
Slepian, D., 1961. First passage time for a particular Gaussian process. The Annals of Mathematical Statistics, 32(2), pp.~610-612.

\bibitem{Shepp}  
Shepp, L.~A., 1971. First passage time for a particular Gaussian process. The Annals of Mathematical Statistics, 42(3), pp.~946-951.

\end{enbibliography}

\end{document}